\documentclass[12pt]{amsart}
\usepackage{amssymb}
\usepackage{amsrefs}
\usepackage{amsthm}
\usepackage{mathrsfs}
\usepackage{amssymb}
\usepackage{amsfonts}
\usepackage[english]{babel}
\usepackage[T1]{fontenc}
\usepackage[latin1]{inputenc}
\usepackage{fullpage}
\usepackage{amsmath}
\usepackage[all]{xy}
\usepackage{stmaryrd}
\usepackage{hyperref}
\usepackage{cleveref}

\usepackage{cite}
\usepackage{xcolor}
\usepackage{ulem}

\newtheorem{define}{Definition}[section]
\newtheorem{lemma}{Lemma}[section]
\newtheorem{proposition}{Proposition}[section]
\newtheorem{theor}{Theorem}[section]

\newtheorem{remark}{Remark}[section]

\newtheorem{theorem}{Main Theorem}

\newcommand{\bg}{\bar{g}}

\numberwithin{equation}{section}

\makeatletter
\newsavebox{\@brx}
\newcommand{\llangle}[1][]{\savebox{\@brx}{\(\m@th{#1\langle}\)}%
  \mathopen{\copy\@brx\kern-0.5\wd\@brx\usebox{\@brx}}}
\newcommand{\rrangle}[1][]{\savebox{\@brx}{\(\m@th{#1\rangle}\)}%
  \mathclose{\copy\@brx\kern-0.5\wd\@brx\usebox{\@brx}}}
\makeatother

\begin{document}

\title{Comparison of total $\sigma_k$-curvature}

\author{Jiaqi Chen}
\address{School of Electrical Engineering and Automation, Xiamen University of Technology, Xiamen 361024, Fujian, P.R. China.}
\email{chenjiaqi@xmut.edu.cn}

\author{Yufei Shan}
\address{School of Mathematical Sciences, Shanghai Jiao Tong University,  Shanghai 200240, P.R. China.}
\email{universeplane1991@sjtu.edu.cn}

\author{Yinghui Ye}
\address{Department of Mathematics, Sun Yat-sen University, Guangzhou, Guangdong 510275, , P.R. China.}
\email{yeyh8@mail2.sysu.edu.cn}

\thanks{Yinghui Ye is partially supported by National Key R\&D Program of China (2024YFA1015000) and National Natural Science Foundation of China (No.12571065); Jiaqi Chen is partially supported by the Scientific Research Foundation of Xiamen University of Technology Under Grant (No.YKJ23009R), the Scientific Research Foundation for Young and Middle aged Teachers in Fujian Province Under Grant (No.JAT231105); Yufei Shan is partially supported by National Natural Science Foundation of China (No.12271348)}


\begin{abstract}
    Volume comparison theorem is a type of fundamental results in Riemannian geometry. In this article, we extend the volume comparison to the comparison of total $\sigma_l$-curvature with respect to $\sigma_k$-curvature. In particular, we prove the comparison holds for metrics close to strictly stable positive Einstein metric when $k$ and $l$ satisfy certain assumptions. As for negative Einstein metrics, we prove a similar comparison result provided certain assumptions on sectional curvature holds for the manifold.
\end{abstract}

\maketitle

\section{Introduction}\label{Section: 1}
As a type of important problem in Riemannian geometry, the comparison problem is usually described as followed: whether certain geometric functional is bounded by the reference manifold provided certain curvature conditions hold? 

The geometric functional that has been studied the most extensively is undoubtedly volume. For example, the famous Bishop-Gromov volume comparison theorem said that the volume of a complete manifold is bounded above by that of the round sphere if its Ricci curvature is bounded below by a corresponding positive constant. People tried to improve this result by replacing the assumption of Ricci curvature by a weaker one. However, it's shown by Bray\cite{Brayphd}, scalar curvature is not sufficient to control the volume for the case of positive curvature. As for the negative curvature, Schoen\cite{Schoen} proposed a conjecture that the volume of a closed manifold is bounded below by the closed hyperbolic manifold if its scalar curvature is bounded below by the corresponding scalar curvature. The evidences supporting this conjecture can be found in \cite{Besson-C-G1,Besson-C-G2,Hamilton,Perelman1,Perelman2}. For non-V-static domains, Corvino, Eichmair, and Miao\cite{Miao} showed that volume comparison does not hold. Inspired by this, Yuan\cite{Yuan_VolumeCW} studied the V-static case with appropriate boundary conditions. In the same paper, Yuan also studied the volume comparison for closed strictly stable Einstein manifold:
\begin{theor}[Yuan, \cite{Yuan_VolumeCW}]\label{theorem: Yuan}
    Suppose $(M^n,\bg)$ is a closed strictly stable Einstein manifold with\begin{equation*}
        \mathrm{Ric}_{\bg}=(n-1)\lambda\bg,
    \end{equation*}
    where $\lambda\neq0$ is a constant. There exists a constant $\epsilon_0>0$ such that for any metric $g$ on $M$ which satisfies\begin{equation*}
        R_g\geq n(n-1)\lambda\ \ \ \mathrm{and}\ \ \ \|g-\bg\|_{C^2(M,\bg)}<\epsilon_0,
    \end{equation*}
    the following volume comparison holds:
    \begin{itemize}
        \item if $\lambda>0$, then $\mathrm{Vol}_{M,g} \leq \mathrm{Vol}_{M,\bg}$;
        \item if $\lambda<0$, then $\mathrm{Vol}_{M,g} \geq \mathrm{Vol}_{M,\bg}$.
    \end{itemize}
    Moreover, the inequality holds in either case if and only if $g$ is isometric to $\bg$.
\end{theor}
As a follow-up study, Lin and Yuan studied the volume comparison with respect to $Q$-curvature for closed strictly stable positive Einstein manifold in \cite{Yuan_Q2}. And Andrade, Cruz and Santos established a similar results with respect to $\sigma_2$-curvature in \cite{sigma2}. Building on these results, Chen, Fang, He and Zhong focused on the $\sigma_k$-curvature, which can be regarded as the generalization of the above curvatures, and obtained the following volume comparison theorem \cite{2025Volume}:

\begin{theor}[Chen-Fang-He-Zhong\cite{2025Volume}]\label{theorem: CFHZ}
    Suppose $(M^n,\bg)$ is a closed strictly stable Einstein manifold with\begin{equation*}
        \mathrm{Ric}_{\bg} = (n-1) \lambda \bg,
    \end{equation*}
    where $\lambda>0$ is a constant. Then there exists a constant $\epsilon_0>0$ such that for any metric $g$ on $M$ satisfying\begin{equation*}
        \sigma_k^g\geq \sigma_k^{\bg}\ \ \ \mathrm{and}\ \ \ \|g-\bg\|_{C^2(M,\bg)}<\epsilon_0,
    \end{equation*}
    the following volume comparison holds:
    \begin{equation*}
        \mathrm{Vol}_{M,g} \leq \mathrm{Vol}_{M,\bg},
    \end{equation*}
    with inequality holds if and only if $g$ is isometric to $\bg$.
\end{theor}

Inspired by treating $\sigma_k$-curvature as a generalization of scalar curvature, we observed that the total $\sigma_k$-curvature is also a generalization of volume, which in particular is the total $\sigma_0$-curvature. Based on this observation, we studied the comparison of total $\sigma_l$-curvature with respect to the $\sigma_k$-curvature. For positive Einstein metric, we proved the following theorem:
\begin{theorem}\label{theorem: main_A}
For $n \geq 3$, suppose $(M^n,\bg)$ is an $n$-dimensional closed strictly stable Einstein manifold with
    \begin{equation*}
        \mathrm{Ric}_{\bg} = (n-1) \lambda \bg,
    \end{equation*} 
    where $\lambda > 0$ is a constant. Given positive integers $k, l \leq n$, there exists a constant $\epsilon>0$, such that for any metric $g$ on $M$ satisfying 
    \begin{equation*}
        \|g-\bg\|_{C^2(M,\bg)}<\epsilon,
    \end{equation*}
    if one of the following assumptions holds:
    \begin{equation*}
        \mathrm{(a)}\ \ \sigma_k^g \geq \sigma_k^{\bg}, \ \ l< \frac{n}{2}\qquad\mathrm{or}\qquad\mathrm{(b)}\ \  \sigma_k^g \leq \sigma_k^{\bg}, \ \ l > \frac{n}{2} + k
    \end{equation*}
    then
    \begin{equation*}
        \int_M\sigma_l^gdv_g\leq\int_M\sigma_l^{\bg}dv_{\bg}.
    \end{equation*}
    Moreover, in both cases, equality holds if and only if $g$ is isometric to $\bg$.
\end{theorem}

As for negative Einstein metric, we proposed certain assumption on sectional curvature:

\begin{theorem}\label{theorem: main_B}
    For $n \geq 3$, suppose $(M^n,\bg)$ is an $n$-dimensional closed strictly stable Einstein manifold with
    \begin{equation*}
        \mathrm{Ric}_{\bg} = (n-1)\lambda\bg,
    \end{equation*} 
    where $\lambda < 0$ is a constant. Given positive integers $k, l \leq n$, assume the sectional curvature $K_{\bg}$ of $\bg$ satisfies
    \begin{equation*}
        K_{\bg}<\frac{n(n-2)}{2\left(n(k-1)-2l(k-l)\right)}\lambda.
    \end{equation*}
    When $l\in\left(\frac{n}{2},  \frac{n}{2}+k \right)$, there exists a constant $\epsilon>0$, such that for any metric $g$ on $M$ satisfying 
    \begin{equation*}
        \|g-\bg\|_{C^2(M,\bg)}<\epsilon,
    \end{equation*}
    if one of the following assumptions holds:
    \begin{equation*}
        \mathrm{(a)}\ \ \sigma_k^g \geq \sigma_k^{\bg}, \ \ k \ \mathrm{is\  odd}\qquad\mathrm{or}\qquad\mathrm{(b)}\ \  \sigma_k^g \leq \sigma_k^{\bg}, \ \ k \ \mathrm{is\  even}
    \end{equation*}
    then
    \begin{align*}
        \mathrm{(i)}\ \ \ \int_M\sigma_l^gdv_g\leq\int_M\sigma_l^{\bg}dv_{\bg}\quad &\mathrm{provided}\quad l\ \mathrm{is\ even;}\\
        \mathrm{(ii)}\ \ \int_M \sigma_l^gdv_g \geq \int_M\sigma_l^{\bg}dv_{\bg}\quad &\mathrm{provided}\quad l\ \mathrm{is\ odd.}
    \end{align*}
    Moreover, in both cases, equality holds if and only if $g$ is isometric to $\bg$.
\end{theorem}

For negative curvature, comparison results hold even for some special $l\notin\left(\frac{n}{2},  \frac{n}{2}+k \right)$, with \Cref{theorem: Yuan} being a known case. Furthermore, we derived the comparison theorem for $k=1$ and $l=1$:

\begin{theorem}\label{theorem: main_C}
    For $n \geq 3$, suppose $(M^n,\bg)$ is an $n$-dimensional closed strictly stable Einstein manifold with\begin{equation*}
    \mathrm{Ric}_{\bg}=(n-1)\lambda\bg,
\end{equation*} 
where $\lambda < 0$ is a constant. There exists a constant $\epsilon>0$, such that for any metric $g$ on $M$ satisfying 
\begin{equation*}
    \|g-\bg\|_{C^2(M,\bg)}<\epsilon, \quad R_g\geq R_{\bg},
\end{equation*}
then
\begin{equation*}
    \int_MR_gdv_g\leq\int_MR_{\bg}dv_{\bg}.
\end{equation*}
Moreover, equality holds if and only if $g$ is isometric to $\bg$.
\end{theorem}

\begin{remark}
    This is not a trivial corollary from \Cref{theorem: Yuan}, which shows that $\mathrm{Vol}_{M,g}\geq\mathrm{Vol}_{M,\bg}$ provided $R_g\geq R_{\bg}$ for negative Einstein metric $\bg$. When $R_g$ is a constant, the above result not only implies that the volume increases, but also gives an estimate of the rate of this increase, that is $\frac{\mathrm{Vol}_{M,g}}{\mathrm{Vol}_{M,\bg}}\geq\frac{R_{\bg}}{R_{g}}$.
\end{remark}

\begin{remark}
    As shown by the above main theorems, our results extend the existing volume comparison theorems. In particular, \Cref{theorem: Yuan} is the case with $k=1,\ l=0$ and \Cref{theorem: CFHZ} is the case of positive curvature with $l=0$.
\end{remark}

\begin{remark}
    Similar to the result in \cite{Yuan_VolumeCW,Yuan_Q2,2025Volume}, the strictly stable condition is necessary. Two counterexamples are constructed in the \Cref{Section: 4}.
\end{remark}

This paper is organized as follow: In \Cref{Section: 2}, we list relevant definitions that will be used in this paper and recall some important geometric variational formulas. In \Cref{Section: 3}, we propose the key functional for the proof and calculate its first and second variations. In \Cref{Section: 4}, we prove our main result and construct a counterexample for unstable Einstein manifold.

\section*{\bfseries Acknowledgment}
The authors would like to thank Professor Yuan Wei for helpful discussions and comments. The authors also would like to express their appreciations to the  anonymous referee for his/her thorough and careful review of the manuscript and providing many detailed and valuable comments and suggestions.

\section{Preliminary}\label{Section: 2}
This section collects frequently used notation and conventions for the reader's convenience. Throughout this article, $(M^n, g)$ denotes an n-dimensional closed Riemannian manifold with $n \geq 3$, unless otherwise stated. All calculations are carried out with respect to the reference metric $\bg$, unless specified otherwise.

\subsection{Basic concepts}

We begin by recalling the concepts of Einstein manifolds.
\begin{define}[Einstein manifold]
     A Riemannian manifold $(M^n, g)$ is said to be Einstein with Einstein constant $\lambda \in \mathbb{R}$ if its Ricci curvature satisfies
    \begin{equation*}
        \mathrm{Ric}_g = (n-1)\lambda g.
    \end{equation*}
\end{define}

We denote the space of transverse-traceless symmetric 2-tensors by $S_{2,g}^{TT}(M)$ and we define the following Einstein operator
\begin{equation*}
    \Delta_E = \Delta_g + 2\mathrm{Rm}_g
\end{equation*}
on $S_{2,g}^{TT}(M)$, where $\Delta_g=\nabla_i\nabla^i$ and $\mathrm{Rm}_g$ denotes the curvature tensor. The positivity of this operator plays a crucial role in the study of Einstein manifolds. To further analyze the geometric properties of Einstein manifolds, we introduce the notion of stability for Einstein metrics. This concept is closely tied to the spectral properties of the Einstein operator $\Delta_E$ on transverse-traceless tensors.
\begin{define}[Stability of Einstein metric]
    For a closed Einstein manifold $(M^n,g)$ where dimension n is at least 3, metric $g$ is said to be stable if the Einstein operator $\Delta_E$ is non-positive on $S_{2,g}^{TT}(M)\backslash\{0\}$, that is
    \begin{equation*}
        \min \mathrm{spec}_{TT}(-\Delta_E)=\inf_{h\in S_{2,g}^{TT}(M)\backslash\{0\}}\frac{\int_M \langle h,-\Delta_E h \rangle_g dv_g}{\int_M |h|_g^2 dv_g} \geq 0.
    \end{equation*}
    Moreover, $g$ is called strictly stable if the inequality is strict.
\end{define}

For a closed Riemannian manifold $(M^n,g)$, the Schouten curvature is  
\begin{equation*}
    S_g=\mathrm{Ric}_g - \frac{R_g}{2(n-1)}g,
\end{equation*}
where $\mathrm{Ric}_g$ and $R_g$ are Ricci curvature and scalar curvature of $g$ respectively. The $\sigma_k$-curvature is defined as the $k$-th elementary symmetric polynomial of the eigenvalues of $S_g$. Namely,

\begin{define}[$\sigma_k$-curvature]
For a Riemannian manifold $(M^n,g)$, let $S_g$ denote the Schouten tensor and $\{\lambda_i\}_{i=1}^n$ its eigenvalues. The $\sigma_k$-curvature of $g$ is defined as the $k$-th elementary symmetric polynomial of these eigenvalues:
\begin{equation*}
    \sigma_k := \sum_{1\leq i_1 < i_2 < \cdots < i_k \leq n} \lambda_{i_1}\lambda_{i_2}\cdots\lambda_{i_k},
\end{equation*}
where $1 \leq k \leq n$. 
\end{define}
In particular, $\sigma_1$ is the scalar curvature (up to normalization) and $\sigma_n$ gives the determinant of the Schouten tensor. Following Reilly \cite{Robert1}, an equivalent but more practical definition of $\sigma_k$-curvature is
\begin{equation}\label{eq: sigmak}
    \sigma_k=\frac{1}{k!}\delta_{i_1\cdots i_k}^{j_1\cdots j_k}S_{j_1}^{i_1}\cdots S_{j_k}^{i_k},
\end{equation}
where $\delta_{i_1\cdots i_k}^{j_1\cdots j_k}$ is the generalized Kronecker delta defined as 
\begin{equation*}
    \delta_{i_1\cdots i_k}^{j_1\cdots j_k}=
    \begin{cases}
        1, \text{ if } j_1\cdots j_k \text{ are distinct integers with even permutation of } i_1\cdots i_k;\\
        -1, \text{ if } j_1\cdots j_k \text{ are distinct integers with odd permutation of } i_1\cdots i_k;\\
        0, \text{ in other cases.}
    \end{cases}
\end{equation*}

\subsection{Geometric variational formulas}

In the study of Riemannian geometry, the behavior of geometric quantities under variations of the metric tensor plays a pivotal role in understanding the structure and stability of manifolds. These variations are fundamental in diverse areas such as general relativity, geometric analysis, and the calculus of variations, where they inform the analysis of critical points, stability conditions, and geometric flows.

In this section, we will list some variational formulas that will be used later. We emphasize that any unspecified calculations are performed in the reference metric $\bg$ and assume that the linearization of the metric is $h\in\mathcal{S}_{2,\bg}(M)$, which is the spaces of all symmetric 2-tensors on $M$.

\subsubsection{Basic variational formulas}

In this part, we provide basic variational formulas for the inverse metric, volume, Ricci curvature and scalar curvature. Detailed proof can be found in \cite{Yuan_ph_D, Yuan_VolumeCW}.  First of all, we recall the variational formulas of the inverse metric:
\begin{lemma}\label{Lemma: variation_of_metrics}
    In local coordinates, the first and second variations of the inverse metric are
    \begin{equation*}
        \dot{g}^{jk}=-h^{jk}
    \end{equation*}
    and
    \begin{equation*}
        \ddot{g}^{jk}=2h_i^jh^{ik}.
    \end{equation*}
\end{lemma}
For volume:
\begin{lemma}\label{Lemma: variation_of_volume}
    The first and second variations of volume are
    \begin{equation*}
        D \mathrm{Vol}_{M,\bg} \cdot h=\frac{1}{2}\int_M \left( tr_{\bg} h \right) dv_{\bg}
    \end{equation*}
    and
    \begin{equation*}
        D^2 \mathrm{Vol}_{M,\bg} \cdot (h,h) = \frac{1}{4} \int_M \left( \left( tr_{\bg}h \right)^2-2|h|^2 \right) dv_{\bg}.
    \end{equation*}
\end{lemma}
For Ricci curvature:
\begin{lemma}\label{Lemma: variation_of_Ricci}
    The first variations of Ricci curvature is
    \begin{equation*}
        D \mathrm{Ric}_{\bg}\cdot h=-\frac{1}{2}\big(\Delta_Lh+\mathscr{L}_X\bg\big),
    \end{equation*}
    where $\Delta_L$ is the Lichnerowicz Laplacian defined by
    \begin{equation*}
        \Delta_L h_{jk}=\Delta_E h_{jk} - R_{ji} h_k^i-R_{ki}h_j^i
    \end{equation*}
    in local coordinates and $X=\big(\frac{1}{2}d(trh)+\delta h\big)^{\#}$ for $(\delta h)_j=-\nabla^ih_{ij}$. The second variations of Ricci curvature, in local coordinates, is
    \begin{equation*}
    \begin{split}
        D^2 \mathrm{Ric}^{\bg}_{jk}\cdot(h,h)=&h^{pi}(R_{ijkl}h_p^l+R_{ijpl}h_k^l-\nabla_i\nabla_kh_{jp}+\nabla_i\nabla_ph_{jk}+\nabla_j\nabla_kh_{ip}-\nabla_j\nabla_ph_{ik})\\
        &+\frac{1}{2}(\nabla_jh_k^p+\nabla_kh_j^p-\nabla^ph_{jk})\big(2(\delta h)_p+\nabla_p(trh)\big)\\
        &+\frac{1}{2}(\nabla_ih_k^p+\nabla_kh_i^p-\nabla^ph_{ik})(\nabla_jh_p^i-\nabla_ph_j^i+\nabla^ih_{jp}).
    \end{split}
    \end{equation*}
\end{lemma}
For scalar curvature:
\begin{lemma}\label{Lemma: variation_of_scalar}
    The first and second variations of scalar curvature are
    \begin{equation*}
        DR_{\bg}\cdot h=\gamma_{\bg} h:=-\Delta(trh)+\delta^2h-\mathrm{Ric}\cdot h
    \end{equation*}
    and
    \begin{equation*}
       \begin{split}
        D^2R_{\bg}\cdot(h,h)=&-2\gamma_{\bg}(h\times h)-\Delta(|h|^2)-\frac{1}{2}|\nabla h|^2-\frac{1}{2}|d(trh)|^2\\
        &+2\langle h,\nabla^2(trh)\rangle-2\langle\delta h,d(trh)\rangle+\nabla_{i}h_{jk}\nabla^jh^{ik},
       \end{split}
    \end{equation*}
    where $(h\times h)_{ij}=h_i^kh_{kj}$
\end{lemma}

\subsubsection{Variations for $\sigma_k$-curvature when $\bg$ is Einstein}

In this subsection, we derive the variation formula for the $\sigma_k$-curvature under the assumption that the reference metric $\bg$ is an Einstein metric with Einstein constant $\lambda$. For such metrics, the Schouten tensor $S_{\bg}$ and the $\sigma_k$-curvature are explicitly given by:
\begin{equation*}
    S_{\bg} = \frac{n - 2}{2} \lambda \bg, \quad \sigma_k^{\bg} = \left( \frac{n - 2}{2} \lambda \right)^k \binom{n}{k}.
\end{equation*}
This simplification arises because the eigenvalues of $S_{\bg}$ are all equal to $\frac{n-2}{2} \lambda$, reducing the $\sigma_k$-curvature to a power of the constant $\lambda$ multiplied by the binomial coefficient $\binom{n}{k}$. These expressions will serve as foundational components in the computation of the variation formula in the subsequent analysis.

For the variations of $\sigma_k$-curvature, instead of using the direct definition, the equivalent expression \Cref{eq: sigmak} will be more convenient for calculation. Also, we will use simplified symbols, such as $\dot{R}$ and $\ddot{R}$, to denotes the first and second variations of the corresponding geometric functional. 

The first variation of $\sigma_k$-curvature is
\begin{equation*}
\begin{split}
    D\sigma_k^{\bg}\cdot h&=\frac{1}{k!}\delta_{i_1\cdots i_k}^{j_1\cdots j_k}S_{j_1}^{i_1}\cdots S_{j_{k-1}}^{i_{k-1}}\times k(\dot{\bg}^{i_kp_k}S_{j_kp_k}+\bg^{i_kp_k}\dot{S}_{j_kp_k})\\
    &=\frac{1}{(k-1)!}\left(\frac{n-2}{2}\lambda\right)^{k-1}\delta_{i_1\cdots i_k}^{j_1\cdots j_k}\delta_{j_1}^{i_1}\cdots \delta_{j_{k-1}}^{i_{k-1}}\left(-\frac{n-2}{2}\lambda h_{j_k}^{i_k}+\bg^{i_kp_k}\dot{S}_{j_kp_k}\right).
\end{split}
\end{equation*}
Using the contraction rule of the generalized Kronecker delta:
\begin{equation*}
    \delta_{i_1\cdots i_p}^{j_1\cdots j_p}\delta_{i_1\cdots i_k}^{j_1\cdots j_k}=p!\frac{(n-k+p)!}{(n-k)!}\delta_{i_{p+1}\cdots i_k}^{j_{p+1}\cdots j_k}
\end{equation*}
for $p<k$, we have
\begin{equation*}
    D\sigma_k^{\bg}\cdot h=\binom{n-1}{k-1}\left(\frac{n-2}{2}\lambda\right)^{k-1}\left(-\frac{n-2}{2}\lambda tr_{\bg} h+tr_{\bg} \dot{S}\right),
\end{equation*}
where
\begin{equation*}
    \dot{S} = \dot{\mathrm{Ric}}-\frac{\dot{R}}{2(n-1)}\bg-\frac{R}{2(n-1)}h.
\end{equation*}
Combining \Cref{Lemma: variation_of_Ricci} and \Cref{Lemma: variation_of_scalar}, we deduce that
\begin{lemma}\label{Lemma: 1st_variation_of_sigma}
    If $\bg$ is an Einstein metric with Einstein constant $\lambda$, then the first variation of the $\sigma_k$-curvature is given by:
    \begin{equation*}
        D\sigma_k^{\bg}\cdot h=\frac{1}{n-1}\binom{n-1}{k-1}\left(\frac{n-2}{2}\right)^k\lambda^{k-1}\big(-\Delta trh+\delta^2h-(n-1)\lambda trh\big).
    \end{equation*}
\end{lemma}

Similarly, for the second derivatives, we have
\begin{align*}
        & D^2\sigma_k^{\bg}\cdot(h,h)\\
        =&\binom{n-1}{k-1} \left(\frac{n-2}{2}\lambda\right)^{k-1} \left(\ddot{\bg}^{ij}S_{ij} + 2\dot{\bg}^{ij}\dot{S}_{ij} + tr_{\bg} \ddot{S} \right)\\
        & + \binom{n-2}{k-2} \left(\frac{n-2}{2}\lambda \right)^{k-2} \delta_{ik}^{jl} \left( \dot{\bg}^{ip} \dot{\bg}^{kq} S_{jp}S_{lq} + 2\dot{\bg}^{ip} \bg^{kq}S_{jp} \dot{S}_{lq} + \bg^{ip} \bg^{kq} \dot{S}_{jp} \dot{S}_{lq} \right)\\
        =&\binom{n-1}{k-1}\left( \frac{n-2}{2} \lambda \right)^{k-1} \left( (n-2)\lambda|h|^2-2h^{ij}\dot{S}_{ij}+tr\ddot{S} \right) \\
        &+\binom{n-2}{k-2} \left( \frac{n-2}{2}\lambda \right)^{k-2} \biggl( \left(\frac{n-2}{2} \lambda \right)^2 \left( \left( tr_{\bg}h \right)^2 - |h|^2\right) \\
        & -(n-2)\lambda\left((tr_{\bg}h) (tr_{\bg}\dot{S}) - h^{ij} \dot{S}_{ij}\right)+\left( (tr_{\bg}\dot{S})^2 - |\dot{S}|^2 \right) \biggr).
    \end{align*}
Observed that for arbitrary Riemannian metric $g$,
\begin{equation*}
    tr_{g}S_g  =  \frac{n-2}{2(n-1)} R_g.
\end{equation*}
Thus
\begin{equation*}
    \frac{n-2}{2(n-1)}\gamma_{\bg}h=\dot{g}^{ij}S_{ij}+\bg^{ij}\dot{S}_{ij}=-\frac{n-2}{2}\lambda tr_{\bg} h+tr_{\bg} \dot{S}.
\end{equation*}
Therefore, we deduce
\begin{lemma}\label{Lemma: 2nd_variation_of_sigma}
    If $\bg$ is an Einstein metric with Einstein constant $\lambda$, then the second variation of the $\sigma_k$-curvature is given by:
    \begin{equation}
    \begin{split}
        D^2 \sigma_k^{\bg} &  \cdot (h,h)
        = \binom{n-1}{k-1} \left(\frac{n-2}{2}\lambda\right)^{k-1} \\
        &\times \bigg( tr_{\bg} \ddot{S} - \frac{2(k-1)}{\lambda(n-1)(n-2)}|\dot{S}|^2-\frac{2(n-k)}{n-1}h^{ij}\dot{S}_{ij}\\
        &\qquad+\frac{(k-1)(n-2)}{2\lambda(n-1)^3}(\gamma_{\bg}h)^2+\frac{(2n-k-1)(n-2)}{2(n-1)}\lambda|h|^2\bigg),
    \end{split}
    \end{equation}
    where
    \begin{equation*}
        \ddot{S}=\ddot{\mathrm{Ric}}-\frac{\ddot{R}}{2(n-1)}\bg-\frac{\dot{R}}{n-1}h.
    \end{equation*}
    and
    \begin{equation*}
        \gamma_{\bg} h := -\Delta_{\bg} \left( tr_{\bg} h\right) + \delta_{\bg}^2 h - \mathrm{Ric}_{\bg} \cdot h
    \end{equation*}
    is the linearization of scalar curvature.
\end{lemma}

\section{The key functional and variations}\label{Section: 3}
For the volume comparison theorems with respect to scalar curvature \cite{Yuan_VolumeCW}, Q-curvature \cite{Yuan_Q2} or $\sigma_k$-curvature \cite{2025Volume}, we can see that a corresponding functional is important for the proof of the theorem. As a generalization of the volume comparison theorem, we propose a generalized functional:
\begin{equation*}
    \mathcal{F}_{\bg}(g)=\left(\int_M\sigma_l^gdv_g\right)^{2k}\left(\int_M\sigma_k^gdv_{\bg}\right)^{n-2l}
\end{equation*}
where $l,k\leq n$ and $g$ is another metric on $M^n$. It is easy to check that $\mathcal{F}_{\bg}(g)$ is a scaling invariant with respect to $g$. 

\begin{remark}
    When $l=\frac{n}{2}$, the functional $\mathcal{F}_{\bg}$ becomes independent of  $\sigma_k$, making it inapplicable for analyzing this case. On the other hand, as suggested in \cite{Chern}, the Gauss-Bonnet-Chern formula implies that the half-dimensional case may be intrinsically related to the topology of the manifold. Consequently, this remains an intriguing problem worthy of further investigation.
\end{remark}

As a generalized functional designed for Einstein manifolds, $\mathcal{F}_{\bg}$ should satisfy the fundamental requirement regarding its first variation:

\begin{proposition}\label{Proposition: critical_pts}
The Einstein metric $\bg$ is a critical point of  $\mathcal{F}_{\bg}$.
\end{proposition}

\begin{proof}
    The first variation of $\mathcal{F}_{\bg}$ at $\bg$ is
    \begin{equation*}
    \begin{split}
        D\mathcal{F}_{\bg}\cdot h=
        &2k \left( \int_M \sigma_l^{\bg} dv_{\bg} \right)^{2k-1} \left(\int_M \sigma_k^{\bg} dv_{\bg} \right)^{n-2l} \times \left(\int_M(D\sigma_l^{\bg}\cdot h) dv_{\bg}+\int_M\sigma_l^{\bg} (Ddv_{\bg}\cdot h)\right)\\
        &+ \left(n-2l\right) \left(\int_M \sigma_l^{\bg}dv_{\bg}\right)^{2k} \left(\int_M \sigma_k^{\bg}dv_{\bg}\right)^{n-2l-1}\int_M \left( D\sigma_k^{\bg}\cdot h \right)dv_{\bg}.
    \end{split}
    \end{equation*}
    From \Cref{Lemma: 1st_variation_of_sigma},
    \begin{equation*}
        \int_M \left( D\sigma_k^{\bg}\cdot h \right)dv_{\bg} = -\binom{n-1}{k-1} \left(\frac{n-2}{2}\lambda\right)^k\int_M\left( tr_{\bg}h \right) dv_{\bg}.
    \end{equation*}
    Therefore, combining the above result and \Cref{Lemma: variation_of_volume} we have
    \begin{equation*}
    \begin{split}
        D\mathcal{F}_{\bg}\cdot h=
        &\left(\mathrm{Vol}_{M,\bg}\right)^{n+2k-2l-1} \left(\frac{n-2}{2}\lambda \right)^{nk}\binom{n}{k}^{n-2l}\binom{n}{l}^{2k}\\
        &\times \left(2k\left(-\frac{l}{n}+\frac{1}{2}\right)-(n-2l)\frac{k}{n} \right) \int_M(tr_{\bg}h) dv_{\bg}=0.
    \end{split}
    \end{equation*}
\end{proof}

For the second variation, to minimize the calculation, we take the TT-gauge condition for symmetric 2-tensor $h$:
\begin{equation*}
    h=\mathring{h}+\frac{trh}{n}\bg\in S^{TT}_{2,\bg}(M)\oplus(C^\infty(M)\cdot\bg)
\end{equation*}
where $\mathring{h}\in S_{2,\bg}^{TT}(M)$. With this decomposition, we have
\begin{equation*}
       |h|^2=|\mathring{h}|^2+\frac{(trh)^2}{n}\ \ \ \mathrm{and}\ \ \  \delta h=-\frac{d(trh)}{n},
\end{equation*}
and we have the following proposition.

\begin{proposition}\label{Proposition: 2nd_variation_of_functional}
For Einstein metric $\bg$ with constant $\lambda$ and $$h = \mathring{h} + \frac{tr_{\bg}h}{n} \bg \in S^{TT}_{2,\bg}(M) \oplus(C^\infty(M) \cdot\bg),$$ we have 
\begin{equation*}
\begin{split}
    D^2\mathcal{F}_{\bg}& \cdot(h,h) =a(n,k,l)\lambda^{nk-1}\\
    \times\bigg[&\frac{n(k-1)-2l(k-l)}{\lambda(n-1)(n-2)}\int_M\mathring{h} \cdot \Delta_E \left(\Delta_E-\frac{n(n-2)^2}{2\left(n(k-1)-2l(k-l)\right)}\lambda\right) \mathring{h} dv_{\bg}\\
    &+\frac{(n-2)(n-2l)(n+2k-2l)}{2n^2} \int_M \left(|dtr_{\bg}h|^2-n\lambda \left(tr_{\bg} h-\overline{tr_{\bg}h} \right)^2\right)dv_{\bg}\bigg],
\end{split}
\end{equation*}
where 
\begin{align*}
    a(n,k,l) = -\frac{1}{2}\left(\frac{n-2}{2}\right)^{nk-1} \binom{n-1}{k-1} \binom{n}{k}^{n-2l-1} \binom{n}{l}^{2k} \left( \mathrm{Vol}_{M,\bg} \right)^{n+2k-2l-1} < 0
\end{align*}
and $\overline{tr_{\bg}h}$ is the mean value of $tr_{\bg}h$ in $M$:
\begin{equation*}
    \overline{trh}:=\frac{1}{\mathrm{Vol}_{M,\bg}}\int_M \left( tr_{\bg}h \right) dv_{\bg}.
\end{equation*}
\end{proposition}

\begin{proof}
   The second variation of $\mathcal{F}_{\bg}$ at $\bg$ is
   \begin{align*}
       &D^2\mathcal{F}_{\bg}\cdot(h,h)\\
       =&2k(2k-1) \left(\int_M\sigma_l^{\bg}dv_{\bg}\right)^{2k-2} \left(\int_M\sigma_k^{\bg}dv_{\bg}\right)^{n-2l} \times\left(\int_M\big(D\sigma_l^{\bg}\cdot h)dv_{\bg}+\int_M\sigma_l^{\bg}(Ddv_{\bg}\cdot h)\right)^2\\
       &+2k\left(\int_M\sigma_l^{\bg}dv_{\bg}\right)^{2k-1}\left(\int_M\sigma_k^{\bg}dv_{\bg}\right)^{n-2l}\\
       &\quad\times\left(\int_MD^2\sigma_l^{\bg}\cdot(h,h)dv_{\bg}+2\int_M(D\sigma_l^{\bg}\cdot h)(Ddv_{\bg}\cdot h)+\int_M\sigma_l^{\bg}(D^2dv_{\bg}\cdot(h,h))\right)\\
       &+4k(n-2l)\left(\int_M\sigma_l^{\bg}dv_{\bg}\right)^{2k-1}\left(\int_M\sigma_k^{\bg}dv_{\bg}\right)^{n-2l-1}\\
       &\quad\times\left(\int_M(D\sigma_l^{\bg}\cdot h)dv_{\bg}+\int_M\sigma_l^{\bg}(Ddv_{\bg}\cdot h)\right)\int_M\big(D\sigma_k^{\bg}\cdot h\big)dv_{\bg}\\
       &+(n-2l)(n-2l-1)\left(\int_M\sigma_l^{\bg}dv_{\bg}\right)^{2k}\left(\int_M\sigma_k^{\bg}dv_{\bg}\right)^{n-2l-2}\left(\int_M\big(D\sigma_k^{\bg}\cdot h\big)dv_{\bg}\right)^2\\
       &+(n-2l)\left(\int_M\sigma_l^{\bg}dv_{\bg}\right)^{2k}\left(\int_M\sigma_k^{\bg}dv_{\bg}\right)^{n-2l-1}\int_M\big(D^2\sigma_k^{\bg}\cdot (h,h)\big)dv_{\bg}.
   \end{align*}
   From \Cref{Lemma: 2nd_variation_of_sigma},
   \begin{equation*}
    \begin{split}
        \int_MD^2\sigma_k^{\bg}&\cdot(h,h)dv_{\bg}=\binom{n-1}{k-1}\left(\frac{n-2}{2}\lambda\right)^{k-1}\\
        &\times\bigg(\int_M(tr_{\bg}\ddot{S})dv_{\bg}-\frac{2(k-1)}{\lambda(n-1)(n-2)}\int_M|\dot{S}|^2dv_{\bg}-\frac{2(n-k)}{n-1}\int_Mh^{ij}\dot{S}_{ij}dv_{\bg}\\
        &\qquad+\frac{(k-1)(n-2)}{2\lambda(n-1)^3}\int_M(\gamma_{\bg}h)^2dv_{\bg}+\frac{(2n-k-1)(n-2)}{2(n-1)}\lambda\int_M|h|^2dv_{\bg}\bigg),
    \end{split}
   \end{equation*}

   Since $h=\mathring{h}+\frac{trh}{n}\bg\in S^{TT}_{2,\bg}(M)\oplus(C^\infty(M)\cdot\bg)$, after rewriting \Cref{Lemma: variation_of_volume,Lemma: variation_of_Ricci,Lemma: variation_of_scalar} under this decompositions, we get (for detailed calculations, please refer to \cite{2025Volume})
    \begin{align*}
        \int_M tr_{\bg} \ddot{S} dv_{\bg} &= \frac{1}{2} \int_M \left[ \mathring{h} \cdot \left(-\frac{3n-2}{2(n-1)}\Delta_E + (n-2)\lambda \right) \mathring{h} - \frac{(n-2)^2}{2n^2} \left| d tr_{\bg} h \right|^2 \right] dv_{\bg};\\
        \int_M \left| \dot{S} \right|^2 dv_{\bg} &= \frac{1}{4} \int_M \left[ \mathring{h} \cdot \left( \Delta_E-(n-2)\lambda \right)^2 \mathring{h} \right]  + \left( \frac{n-2}{n}\right)^2 \left[ \left(\Delta tr_{\bg} h \right)^2 - (n-1) \lambda \left| dtr_{\bg}h \right|^2 \right] dv_{\bg};\\
        \int_M h^{ij} \dot{S}_{ij} dv_{\bg} &= \frac{1}{2} \int_M \left[ \mathring{h} \cdot \left(-\Delta_E+(n-2)\lambda \right) \mathring{h}+\frac{n-2}{n^2} \left| dtr_{\bg}h \right|^2 \right] dv_{\bg};\\
        \int_M (\gamma_{\bg}h)^2 dv_{\bg} &= (n-1)^2 \int_M \left( \frac{1}{n^2} \left(\Delta tr_{\bg}h \right)^2 - \frac{2}{n} \lambda \left| dtr_{\bg}h \right|^2+\lambda^2 \left| tr_{\bg}h \right|^2 \right) dv_{\bg}.
    \end{align*}
   Combining these expressions, we have
   \begin{equation*}
   \begin{split}
       \int_MD^2\sigma_k^{\bg}&\cdot(h,h)dv_{\bg}=\binom{n-1}{k-1}\left(\frac{n-2}{2}\lambda\right)^{k-1}\\
       &\times\int_M\bigg[\frac{k-1}{2\lambda(n-1)(n-2)}\mathring{h}\cdot\Delta_E\left(-\Delta_E+\frac{\lambda(n-2)^2}{2(k-1)}\right)\mathring{h}+\frac{n-2}{2}\lambda|\mathring{h}|^2\\
       &\qquad\quad+\frac{(n-2)(n+2k)}{4n^2}\left(-|dtrh|^2+\frac{2n(k+1)}{n+2k}\lambda|trh|^2\right)\bigg]dv_{\bg}.
   \end{split}
   \end{equation*}

   Finally, we get the formula of $D^2\mathcal{F}_{\bg}({\bg})\cdot(h,h)$ by inserting the above result back into the initial expression.
\end{proof}

\section{Total {$\sigma_l$-curvature} comparison with respect to {$\sigma_k$-curvature}}\label{Section: 4}

To prove our comparison theorem, we examine the second variation of $\mathcal{F}_{\bg}$ at $\bg$. \Cref{Proposition: 2nd_variation_of_functional} reveals that this depends on the spectral analysis of two key operators: the elliptic operator $\Delta_E$ restricted to $S_{2,\bg}^{TT}(M)$ and the Laplace-Beltrami operator acting on the space of functions.
\subsection{Case 1: $k=1$}\label{Subsection: 4.1}
~\\
The second order variation for the case where $k=1$, is deduced to
\begin{equation*}
    \begin{split}
    D^2\mathcal{F}_{\bg}&\cdot(h,h)= a(n,1,l)\lambda^{n-1}\\
    \times\Bigg[&\frac{2l(l-1)}{\lambda(n-1)(n-2)}\int_M \mathring{h} \cdot \Delta_E\left(\Delta_E-\frac{n(n-2)^2}{4l(l-1)}\lambda\right) \mathring{h} dv_{\bg}\\
    & +\frac{(n-2)(n-2l)\left(n-2(l-1)\right)}{2n^2} \int_M \left( \left| dtr_{\bg}h \right|^2-n \lambda(tr_{\bg}h - \overline{tr_{\bg}h})^2 \right) dv_{\bg} \Bigg].
    \end{split}
\end{equation*}
The case where $l=0$ is already studied by Yuan in \Cref{theorem: Yuan}.

When $l=1$, the second order variation is further deduced to
\begin{equation*}
    \begin{split}
    D^2\mathcal{F}_{\bg}\cdot(h,h)= a(n,1,1)\lambda^{n-1}\Bigg[&-\frac{n(n-2)}{2(n-1)}\int_M \mathring{h} \cdot \Delta_E \mathring{h} dv_{\bg}\\
    & +\frac{(n-2)^2}{2n} \int_M \left( \left| dtr_{\bg}h \right|^2-n \lambda(tr_{\bg}h - \overline{tr_{\bg}h})^2 \right) dv_{\bg} \Bigg].
    \end{split}
\end{equation*}

We recall the well-known Lichnerowicz-Obata theorem \cite{PeterLi}:
\begin{lemma}[Lichnerowicz-Obata]\label{Lemma: Lichnerowicz_Obata}
    Suppose $(M^n,\bg)$ is an $n$-dimensional closed Riemannian manifold with Ricci curvature tensor
    \begin{equation*}
        \mathrm{Ric}_{\bg} \geq (n-1)\lambda \bg,
    \end{equation*}
    where $\lambda>0$ is a constant. Then for any function $u\in C^{\infty}(M)$ that is not identically a constant, we have
    \begin{equation*}
        \int_M |du|^2 dv_{\bg} \geq \lambda n \int_M (u-\bar{u})^2 dv_{\bg},
    \end{equation*}
    where equality holds if and only if $(M^n,\bg)$ is isometric to $\mathbb{S}^n \left(\frac{1}{\sqrt{\lambda}}\right)$, which is the round sphere with radius $\frac{1}{\sqrt{\lambda}}$, and $u$ is a first eigenfunction of the Laplace-Beltrami operator.
\end{lemma}

For strictly stable Einstein metric $\bg$, applying the above lemma, we immediately get 
\begin{itemize}
        \item if $\lambda>0$, then
    \begin{equation*}
    D^2 \mathcal{F}_{\bg} \cdot (h,h) \leq 0;
    \end{equation*}
    \item if $\lambda<0$ and
    \begin{enumerate}
        \item $n$ is odd, then
        \begin{equation*}
        D^2\mathcal{F}_{\bg}\cdot(h,h) \leq 0;
        \end{equation*}
    \item $n$ is even, then
        \begin{equation*}
            D^2\mathcal{F}_{\bg}\cdot(h,h) \geq 0.
        \end{equation*}
    \end{enumerate}
    \end{itemize}

For $l\geq2$, it follows exactly the same approach as that for $k \geq 2$ presented below. To avoid redundant repetition, we shall focus exclusively on the case of $k \geq 2$ in the subsequent discussion.

\subsection{Case 2: $k\geq2$}
\ \newline
Firstly, we rewrite the second variation as 
\begin{equation*}
    D^2\mathcal{F}_{\bg}\cdot(h,h)=a(n,k,l)\lambda^{nk-1}\left(\frac{\alpha}{\lambda}I(\mathring{h},\mathring{h})+\beta J(trh,trh)\right),
\end{equation*}
where 
\begin{equation*}
    I(\mathring{h},\mathring{h}) = \int_M \mathring{h} \cdot \Delta_E \left( \Delta_E-\mu\lambda \right) \mathring{h} dv_{\bg}
\end{equation*}
and
\begin{equation*}
    J(trh,trh)=\int_M \left( |d tr_{\bg} h|^2-n\lambda(tr_{\bg}h-\overline{tr_{\bg}h})^2\right)dv_{\bg}
\end{equation*}
with 
\begin{equation*}
    \alpha=\frac{n(k-1)-2l(k-l)}{(n-1)(n-2)}, \quad \beta=\frac{(n-2)(n-2l)(n+2k-2l)}{2n^2}, \quad \mu=\frac{n(n-2)^2}{2\left(n(k-1)-2l(k-l)\right)}.
\end{equation*}

~\\
\textbf{PART 1: $I(\mathring{h},\mathring{h})$}

Since $n(k-1)-2l(k-l) \geq \frac{k}{2}(n-k) + n(\frac{k}{2}-1)>0$, we have $\mu>0$. Therefore
\begin{equation*}
    I(\mathring{h},\mathring{h}) \geq 0,
\end{equation*}
if the strictly stable Einstein metric $\bg$ satisfies 
\begin{center}
    $\lambda>0$\ \ or\ \ $\Lambda_E:= \min \mathrm{spec}_{TT}(-\Delta_E) \geq -\mu\lambda>0$.
\end{center}

In \cite{Besse2008}, Besse introduced the induced exterior differential $d^D$ on $T^*M$-valued differential forms and proved the Weitzenb$\ddot{o}$ck formula:
\begin{equation*}
    \left( \delta^D d^D + d^D \delta^D \right) \mathring{h} = \left( -\Delta_E + \mathrm{Rm} + (n-1)\lambda \right)\mathring{h},
\end{equation*}
where $\delta^D$ is the adjoint operator. Now if $\mathring{h}$ is the eigentensor of $-\Delta_E$ with eigenvalue $\Lambda$, we have
\begin{equation*}
\begin{split}
    0\leq & \int_M\left(| d^D\mathring{h} |^2 + | \delta^D\mathring{h} |^2\right)dv_{\bg}\\
    =& \int_M\mathring{h}\cdot\left(-\Delta_E+\mathrm{Rm}+(n-1) \lambda \right) \mathring{h}dv_{\bg}\\
    \leq&(\Lambda+\theta+(n-1)\lambda)\int_M|\mathring{h}|^2dv_{\bg}
\end{split}
\end{equation*}
where $\theta$ denotes the largest eigenvalue of $\mathrm{Rm}$ on $S_{2,\bg}^{TT}$, that is,
\begin{equation*}
    \theta=\sup_{0\neq \mathring{h}\in S_{2,\bg}^{TT}}\frac{\int_M\mathrm{Rm}(\mathring{h},\mathring{h})dv_{\bg}}{\int_M|\mathring{h}|^2dv_{\bg}}
\end{equation*}
Thus if $\Lambda<-\theta-(n-1)\lambda$, we have $h=0$. That is to say\begin{equation*}
    \Lambda_E\geq-\theta-(n-1)\lambda.
\end{equation*} 
Furthermore, we recall the Algebraic Lemma\cite{Fujitani1979COMPACTSP}:
\begin{lemma}[Algebraic Lemma]
    For Einstein manifold $(M^n,\bg)$, Denote the largest eigenvalue of $\mathrm{Rm}_{\bg}$ on trace-free symmetric 2-tensor by $a_0$, that is 
    \begin{equation*}
       a_0=\sup_{0\neq h:\ tr_{\bg}h=0}\frac{\int_M\mathrm{Rm}(h,h)dv_{\bg}}{\int_M|h|^2dv_{\bg}}.
    \end{equation*}
    Then
    \begin{equation*}
        a_0\leq\min\left\{(n-2)K_{max}-\frac{R}{n},\frac{R}{n}-nK_{min}\right\}
    \end{equation*}
    where $K_{max}$ and $K_{min}$ are the maximum and minimum of the sectional curvature.
\end{lemma}
When $\lambda<0$, this lemma implies that $$\theta\leq a_0\leq(n-2)K_{max}-\frac{R}{n}=(n-2)K_{max}-(n-1)\lambda.$$ Thus,
\begin{equation*}
    \Lambda_E\geq-(n-2)K_{max}.
\end{equation*} 
Therefore $$\Lambda_E>-\mu \lambda$$ provided $$K_{max}<\frac{\mu}{n-2}\lambda=\frac{n(n-2)}{2\left(n(k-1)-2l(k-l)\right)}\lambda.$$ To sum up, since $\alpha>0$, we have
\begin{equation}\label{eq: part1}
    \frac{\alpha}{\lambda} I(\mathring{h},\mathring{h})
    \begin{cases}
        &\geq 0, \qquad \mathrm{if} \quad \lambda>0\\
        &\leq 0, \qquad \mathrm{if} \quad K_{\bg} < \frac{n(n-2)}{2\left(n(k-1)-2l(k-l)\right)}\lambda<0
    \end{cases}
\end{equation}
where $K_{\bg}$ is the sectional curvature with respect to $\bg$. Moreover, when the equality holds. If $\mathring{h}\neq0$, we have
\begin{equation*}
    -\int_M\mathring{h}\cdot\Delta_E\mathring{h}dv_{\bg}=-\mu\lambda\int_M|\mathring{h}|^2dv_{\bg}\begin{cases}
        <0\qquad \qquad \ \  \ \quad  \mathrm{for}\ \ \lambda>0\\
        <\Lambda_E\int_M|\mathring{h}|^2dv_{\bg}\ \ \mathrm{for}\ \ \lambda<0.
    \end{cases}
\end{equation*}
This is contradict to the strict stability of $\bg$ and the definition of $\Lambda_E$, thus $\mathring{h}=0$.

~\\
\textbf{PART 2: $J(trh,trh)$}

Applying \Cref{Lemma: Lichnerowicz_Obata}, we immediately get 
\begin{equation*}
    J(trh,trh) \geq 0.
\end{equation*}
Since $\beta\begin{cases} >0 & \mathrm{if}\ \ l\in\big[0,\frac{n}{2}\big)\cup\big(\frac{n}{2}+k,n\big] \\ <0 & \mathrm{if}\ \ l\in\big(\frac{n}{2},\frac{n}{2}+k\big) \end{cases}$
, we have
\begin{equation}\label{eq: part2}
    \beta J(trh,trh)
    \begin{cases}
        &\geq 0 \qquad \mathrm{if} \ \ l\in\big[0,\frac{n}{2}\big)\cup\big(\frac{n}{2}+k,n\big]\\
        &\leq 0 \qquad \mathrm{if} \ \ l\in\big(\frac{n}{2},\frac{n}{2}+k\big).
    \end{cases}
\end{equation}
The equality in either case holds if and only if 
\begin{itemize}
    \item $tr_{\bg} h \equiv \overline{tr_{\bg}h}$ when $\lambda<0$ or
    \item $tr_{\bg}h$ is a first eigenfunction of the Laplace-Beltrami operator and $(M^n,\bg)$ is isometric to $\mathbb{S}^n \left(\frac{1}{\sqrt{\lambda}}\right)$.
\end{itemize}

\begin{remark}
    When $l=\frac{n}{2}+k$, we have $\beta J(trh,trh)\equiv0$. Therefore, this case can not be handled using the second order variation alone. As higher order variations might provide a solution, we leave this problem for future investigation.
\end{remark}

Finally, combining \Cref{eq: part1,eq: part2}, we obtain:
\begin{proposition}\label{Proposition: Negativity}
Suppose $(M^n,\bg)$ is a strictly stable Einstein manifold with $\lambda \neq 0$ and denote the sectional curvature of $M$ by $K_{\bg}$. Given integral $k\in\left[2,n\right],\ l\in[0,n]$, for any $h\in S_{2,\bg}^{TT} (M)\oplus(C^{\infty}(M) \cdot \bg)$, 
    \begin{itemize}
        \item if $\lambda>0$ and $l\in\big[0,\frac{n}{2}\big)\cup\big(\frac{n}{2}+k,n\big]$, then
    \begin{equation*}
    D^2 \mathcal{F}_{\bg} \cdot (h,h) \leq 0;
    \end{equation*}
    \item if $\lambda<0$ and $K_{\bg} < \frac{n(n-2)}{2\left(n(k-1)-2l(k-l)\right)}\lambda$, $l\in\big(\frac{n}{2},\frac{n}{2}+k\big)$ and
    \begin{enumerate}
        \item $nk$ is even, then
        \begin{equation*}
        D^2\mathcal{F}_{\bg}\cdot(h,h) \leq 0;
        \end{equation*}
    \item $nk$ is odd, then
        \begin{equation*}
            D^2\mathcal{F}_{\bg}\cdot(h,h) \geq 0.
        \end{equation*}
    \end{enumerate}
    \end{itemize}
Moreover, equality in each of the above cases holds if and only if 
\begin{itemize}
    \item $h\in\mathbb{R}\bg$ when $\lambda<0$ or
    \item $h\in E_{n\lambda}\bg$ and $(M^n,\bg)$ is isometric to $\mathbb{S}^n \left( \frac{1}{\sqrt{\lambda}} \right)$, where
    \begin{equation*}
        E_{n\lambda} = \left\{\left.u\in C^{\infty}\left(\mathbb{S}^n\left(\frac{1}{\sqrt{\lambda}}\right)\right)\ \right|\ \Delta_{\mathbb{S}^n\left(\frac{1}{\sqrt{\lambda}}\right)}u +n\lambda u=0\right\}.
    \end{equation*}

\end{itemize} 
\end{proposition}

In order to investigate the local structure of Einstein metrics, we introduce the following slice theorem \cite{BrendleMaruqes, Viaclovsky2016}:
\begin{theor}[Ebin-Palais slice theorem]\label{Theorem: Ebin-Palais slice theorem}
    Suppose $(M^n,\bg)$ is a closed n-dimensional Einstein manifold with $\lambda \in \mathbb{R}$. Let $\mathcal{M}$ be the space of all Riemannian metrics on $M$. There exists a slice $\mathcal{S}_{\bg}$ though $\bg$ in $\mathcal{M}$. That is, For a fixed real number $p>n$, one can find a constant $\epsilon>0$ such that for any metrics $g \in \mathcal{M}$ with $||g-\bg||_{W^{2,p}(M,\bg)} < \epsilon$, there is a diffeomorphism $\varphi$ with $\varphi^*g\in\mathcal{S}_{\bg}$. Moreover, for a smooth local slice $\mathcal{S}_{\bg}$, we have \begin{itemize}
        \item $T_{\bg}\mathcal{S}_{\bg}=S_{2,\bg}^{TT}(M)\oplus(C^{\infty}(M)\cdot\bg)$ when $(M^n,\bg)$ is not isometric to the round sphere;
        \item $T_{\bg}\mathcal{S}_{\bg}=S_{2,\bg}^{TT}(M)\oplus(E_{n\lambda}^{\perp}\cdot\bg)$ when $(M^n,\bg)$ is isometric to $\mathbb{S}^n(\frac{1}{\sqrt{\lambda}})$,\\
        where $E_{n\lambda}^{\perp}=\left\{\left.u\in C^{\infty}\left(\mathbb{S}^n\left(\frac{1}{\sqrt{\lambda}}\right)\right)\ \right|\int_{\mathbb{S}^n\left(\frac{1}{\sqrt{\lambda}}\right)}uvdv_{\bg}=0,\ \forall v\in E_{n\lambda}\right\}$ 
    \end{itemize}
    And
    \begin{equation*}
        S_2(M)=\{\mathcal{L}_{\bg}(X)|X\in\mathfrak{X}(M)\}\oplus T_{\bg}\mathcal{S}_{\bg}.
    \end{equation*}
    
\end{theor}

Applying the slice theorem, we can restrict $\mathcal{F}_{{\bg}}$ on a local slice $\mathcal{S}_{\bg}$ and denote it by $\mathcal{F}_{{\bg}}^{\mathcal{S}}$. To investigate the behavior of $\mathcal{F}_{{\bg}}^{\mathcal{S}}$, we state the following Morse lemma \cite{Fischer1975}:

\begin{lemma}
    (Morse lemma). Let $\mathcal{P}$ be a Banach manifold and $F:\mathcal{P}\rightarrow\mathbb{R}$ a $C^2$-function. Suppose $\mathcal{Q}\subset\mathcal{P}$ is a submanifold satisfying
    \begin{equation*}
        F=0\quad\mathrm{and}\quad dF=0\quad\mathrm{on}\ \mathcal{Q}
    \end{equation*}
    and that there is a smooth normal bundle neighborhood of $\mathcal{Q}$ such that if $\mathcal{E}_x$ is the normal complement to $T_x\mathcal{Q}$ in $T_x\mathcal{P}$ then $d^2F(x)$ is weakly negative definite on $\mathcal{E}_x$, that is 
    \begin{equation*}
        d^2F(x)(v,v)\leq0
    \end{equation*}
    with equality only if $v=0$. Let $\langle,\rangle_x$ be a weak Riemannian structure with a smooth connection and assume that $F$ has a smooth $\langle,\rangle_x$-gradient, $Y(x)$. Assume $DY(x)$ maps $\mathcal{E}_x$ to $\mathcal{E}_x$ and is an isomorphism for $x\in\mathcal{Q}$. Then there is a neighborhood $U$ of $\mathcal{Q}$ such that for any $y\in U$,
    \begin{equation*}
        F(y)\geq0
    \end{equation*}
    implies $y\in\mathcal{Q}$.
\end{lemma}

Applying it to \Cref{Proposition: critical_pts,Proposition: Negativity}, we obtain the following rigidity result:

\begin{proposition}\label{Proposition: Regidity}
    Suppose $(M^n,\bg)$ is a strictly stable Einstein manifold with $\lambda\neq0$ and denote the sectional curvature of $M$ by $K$. Let $\mathcal{S}_{\bg}$ be a local slice through $\bg$. Then there is a neighborhood $U_{\bg}\subset\mathcal{S}_{\bg}$ of $\bg$, such that for any $g\in U_{\bg}$,
    \begin{equation*}
        g=c^2\bg
    \end{equation*}
    with a constant $c>0$ if any of the following is satisfied:
        \begin{itemize}
        \item $\lambda>0$, $l\in\big[0,\frac{n}{2}\big)\cup\big(\frac{n}{2}+k,n\big]$, and
    \begin{equation*}
    \mathcal{F}_{{\bg}}^{\mathcal{S}}(g)\geq\mathcal{F}_{{\bg}}^{\mathcal{S}}(\bg);
    \end{equation*}
    \item $\lambda<0$, $K_{\bg} < \frac{n(n-2)}{2\left(n(k-1)-2l(k-l)\right)}\lambda$, $l\in\big(\frac{n}{2},\frac{n}{2}+k\big)$ and
    \begin{enumerate}
        \item $nk$ is even and
        \begin{equation*}
        \mathcal{F}_{{\bg}}^{\mathcal{S}}(g)\geq\mathcal{F}_{{\bg}}^{\mathcal{S}}(\bg);
        \end{equation*}
    \item $nk$ is odd and
        \begin{equation*}
            \mathcal{F}_{{\bg}}^{\mathcal{S}}(g)\leq\mathcal{F}_{{\bg}}^{\mathcal{S}}(\bg).
        \end{equation*}
    \end{enumerate}
    \end{itemize}
\end{proposition}
\begin{proof}
    Take $\mathcal{P}=\mathcal{S}_{\bg}$ and
    \begin{equation*}
        F=\mathcal{F}_{\bg}^{\mathcal{S}}-\mathcal{F}_{\bg}^{\mathcal{S}}(\bg):\mathcal{P}\rightarrow\mathbb{R}.
    \end{equation*}
    Since $F$ is dilation-invariant, according to Proposition 1, 
    \begin{equation*}
        F=0\quad\mathrm{and}\quad dF=0
    \end{equation*}
    on
    \begin{equation*}
        \mathcal{Q}=\{c^2\bg\in\mathcal{S}_{\bg}|c>0\}.
    \end{equation*}
    Koiso proved in \cite{Koiso1980} that strict stability allows us to take $\mathcal{Q}$ to be the subset of $\mathcal{S}_{\bg}$ consisted of Einstein metrics near $\bg$ without loss of generality.

    The tangent space of $\mathcal{Q}$ at $\bg$ is 
    \begin{equation*}
        T_{\bg}\mathcal{Q}=\mathbb{R}\bg
    \end{equation*}
    and its $L^2$-complement in $T_{\bg}\mathcal{P}$ is given as follow:
    \begin{itemize}
        \item when $\bg$ is not isometric to the round sphere: $$\mathcal{E}_{\bg}=S_{2,\bg}^{TT}(M)\oplus \left\{u\bg \left|u\in C^{\infty}(M)\ \ \mathrm{satisfying}\ \int_M udv_{\bg}=0 \right. \right\};$$ 
        \item when $\bg$ is isometric to $\mathbb{S}^n(\frac{1}{\sqrt{\lambda}})$: $$\mathcal{E}_{\bg}=S_{2,\bg}^{TT}(M)\oplus \left\{u\bg \left|u\in E_{n\lambda}^{\perp}\ \ \mathrm{satisfying}\ \int_M udv_{\bg}=0  \right. \right \}.$$ 
    \end{itemize}

    Define a weak Riemannian structure on $\mathcal{P}$: for any $g\in\mathcal{P}$,
    \begin{equation*}
        \langle\langle h,k\rangle\rangle_g:=\int_M\left(\langle\nabla_gh,\nabla_gk\rangle_g+\langle h,k\rangle_g\right)dv_g=\int_M\langle(-\Delta_g+1)h,k\rangle_gdv_g
    \end{equation*}
    for any $h,k\in T_g\mathcal{P}$. According to \cite{Ebin1970}, it has a smooth connection and the $\langle\langle,\rangle\rangle_g$-gradient of $F$ is 
    \begin{equation*}
    \begin{split}
        Y_g=P_g(-\Delta_g+1)^{-1}\Bigg\{F(g)\bigg[&2k\left(\int_M\sigma_l^gdv_g\right)^{-1}\left(\Gamma_{l,g}^*(1)+\frac{1}{2}\sigma_l^gg\right)\\
        &+(n-2l)\left(\int_M\sigma_k^gdv_{\bg}\right)^{-1}\Gamma_{k,g}^*(f_g)\bigg]\Bigg\}
    \end{split}
    \end{equation*}
    where $P_g$ is the orthogonal projection to $T_g\mathcal{P}$, $\Gamma_{k,g}^*$ is the $L^2$-adjoint of $D\sigma_k(g)$ and $f_g$ is defined by $dv_{\bg}=f_gdv_g$. It is obvious that $Y_g$ is smooth. 
    
    For any $g\in\mathcal{Q}$, the linearization of $Y$ at $\bg$ is 
    \begin{equation*}
    \begin{split}
        DY_{\bg}\cdot h=P_{\bg}(-\Delta_{\bg}+1)^{-1}\left(D^2F\cdot(h,\cdot)\right)
    \end{split}
    \end{equation*}
    for any $h\in\mathcal{E}_{\bg}$. When 
    $$\lambda>0 \quad \mathrm{and} \quad l\in\big[0,\frac{n}{2}\big)\cup\big(\frac{n}{2}+k,n\big]$$ or $$K_{\bg} < \frac{n(n-2)}{2\left(n(k-1)-2l(k-l)\right)}\lambda < 0,\quad l\in\big(\frac{n}{2},\frac{n}{2}+k\big)\quad \mathrm{and}\quad nk\ \mathrm{is\ even,}$$
    since $D^2F$ is strictly negative definite on $\mathcal{E}_{\bg}$, we can see that $DY_{\bg}$ is an isomorphism.

    By Morse Lemma, there exists a neighborhood $U_1$ of $\mathcal{Q}$ such that for any $g\in U_1$, if $\mathcal{F}_{\bg}^{\mathcal{S}}(g)\geq\mathcal{F}_{\bg}^{\mathcal{S}}(\bg)$, then $g\in\mathcal{Q}$, that is
    \begin{equation*}
        g=c^2\bg
    \end{equation*}
    for a constant $c>0$.

    When 
    $$K_{\bg} < \frac{n(n-2)}{2\left(n(k-1)-2l(k-l)\right)}\lambda < 0,\quad l\in\big(\frac{n}{2},\frac{n}{2}+k\big)\quad \mathrm{and}\quad nk\ \mathrm{is\ odd,}$$ 
    take $F=\mathcal{F}_{\bg}^{\mathcal{S}}(\bg)-\mathcal{F}_{\bg}^{\mathcal{S}}$, and in the same way, there exists a neighborhood $U_2$ of $\mathcal{Q}$ such that for any $g\in U_2$, if $\mathcal{F}_{\bg}^{\mathcal{S}}(g)\leq\mathcal{F}_{\bg}^{\mathcal{S}}(\bg)$, then 
    \begin{equation*}
        g=c^2\bg
    \end{equation*}
    for a constant $c>0$.

    By taking $U_{\bg}=U_1\cap U_2$, we finish the proof.
\end{proof}

Now we are ready to prove the main theorem:
\begin{proof}[Proof of Main Theorem]
    According to the slice theorem, we can find a local slice $\mathcal{S}_{\bg}$ through $\bg$. Since $||g - \bg||_{C^2(M,\bg)} < \epsilon$, there is a diffeomorphism $\varphi$ such that $\varphi^*g \in U_{\bg}$ given in \Cref{Proposition: Regidity}. To make it looks concise, we also denote it by $g$.

    Before providing the proofs in each cases, we claim that for all $k$, $\sigma_k^g$ shares the identical sign of $\sigma_k^{\bg}$. This comes form the fact that the Schouten curvature of $g$ is close to the one of $\bg$, owing to the assumption: $||g - \bg||_{C^2(M,\bg)} < \epsilon$. In details, according to the \Cref{Lemma: variation_of_Ricci,Lemma: variation_of_scalar},
    \begin{equation*}
        |R_g-R_{\bg}|+|Ric_{ij}^g-Ric_{ij}^{\bg}|\leq C(|h|+|\nabla_{\bg}h|+|\nabla^2h|)\ \ \ \forall\ i,j
    \end{equation*}
    where, $|h|=\max_{i,j}|h_{ij}|$, $|\nabla_{\bg}h|=\max_{i,j,k}|\nabla^{\bg}_kh_{ij}|$ and $|\nabla^2h|=\max_{i,j,k,l}|\nabla^{\bg}_l\nabla^{\bg}_kh_{ij}|$. By the definition of Schouten curvature,
    \begin{equation*}
        |S_{ij}^g-S_{ij}^{\bg}|\leq C(|h|+|\nabla_{\bg}h|+|\nabla^2h|)\ \ \ \forall\ i,j.
    \end{equation*}
    Using \Cref{eq: sigmak}, we have
    \begin{equation*}
        |\sigma_k^g-\sigma_k^{\bg}|\leq C(|h|+|\nabla_{\bg}h|+|\nabla^2h|)\ \ \ \forall\ i,j.
    \end{equation*}
    Since $||g - \bg||_{C^2(M,\bg)} < \epsilon$, we can take $\epsilon$ small enough to ensure
    \begin{equation*}
        |\sigma_k^g-\sigma_k^{\bg}|\leq \frac{1}{2}|\sigma_k^{\bg}|\ \ \ \forall\ i,j.
    \end{equation*}

    \ \\

    When $\lambda>0$, for all $k$, $\sigma_k^{\bg},\ \sigma_k^g>0$. If one of the following assumption holds:
    \begin{equation*}
        \mathrm{(a)}\ \ \sigma_k^g \geq \sigma_k^{\bg}, \ \ l< \frac{n}{2}\qquad\mathrm{or}\qquad\mathrm{(b)}\ \  \sigma_k^g \leq \sigma_k^{\bg}, \ \ l > \frac{n}{2} + k.
    \end{equation*}
    by assuming reversely
    \begin{equation*}
        \int_M\sigma_l^gdv_g\geq\int_M\sigma_l^{\bg}dv_{\bg},
    \end{equation*}
    we have
    \begin{align*}
        \mathcal{F}_{\bg}^{\mathcal{S}}(g)=&\left(\int_M\sigma_l^gdv_g\right)^{2k}\left(\int_M\sigma_k^gdv_{\bg}\right)^{n-2l} \geq \left( \int_M\sigma_l^{\bg}dv_{\bg} \right)^{2k} \left(\int_M\sigma_k^{\bg}dv_{\bg}\right)^{n-2l}=\mathcal{F}_{\bg}^{\mathcal{S}}(\bg).
    \end{align*}
    Therefore, by \Cref{Proposition: Regidity},
    \begin{equation*}
        g=c^2\bg
    \end{equation*}
    for constant $c>0$. Now
    \begin{align*}
        \int_M\sigma_l^gdv_g &= c^{n-2l}\int_M\sigma_l^{\bg}dv_{\bg} \geq \int_M\sigma_l^{\bg}dv_{\bg}>0,
    \end{align*}
    which implies 
    \begin{equation*}
        \mathrm{(a)}\ \ c \geq 1 \ \ \mathrm{for} \ \ l< \frac{n}{2}\qquad\mathrm{and}\qquad\mathrm{(b)}\ \ c \leq 1 \ \ \mathrm{for} \ \ l > \frac{n}{2} + k.
    \end{equation*}
    On the other hand,
    \begin{align*}
        \sigma_k^g &= c^{-2k}\sigma_k^{\bg}\begin{cases}
            \geq\sigma_k^{\bg}\ \ \mathrm{for}\ \ l<\frac{n}{2}\\
            \leq\sigma_k^{\bg}\ \ \mathrm{for}\ \ l>\frac{n}{2}+k
        \end{cases}
    \end{align*}
    which implies
    \begin{equation*}
        \mathrm{(a)}\ \ c \leq 1 \ \ \mathrm{for} \ \ l< \frac{n}{2}\qquad\mathrm{and}\qquad\mathrm{(b)}\ \ c \geq 1 \ \ \mathrm{for} \ \ l > \frac{n}{2} + k.
    \end{equation*}
    Therefore, in both cases, we have $c=1$. That is, $g=\bg$.

    \ \\

    When $\lambda<0$ and $l\in\left(\frac{n}{2},\frac{n}{2}+k\right)$, if one of the following assumption holds:
    \begin{equation*}
        \mathrm{(a)}\ \ \sigma_k^g \geq \sigma_k^{\bg}, \ \ k \ \mathrm{is\  odd}\qquad\mathrm{or}\qquad\mathrm{(b)}\ \  \sigma_k^g \leq \sigma_k^{\bg}, \ \ k \ \mathrm{is\  even}
    \end{equation*}

    Notice that 
    \begin{equation*}
        \mathrm{(a)}\ \ \sigma_k^g,\ \sigma_k^{\bg}<0 \ \ \mathrm{for}\ \ k \ \ \mathrm{is\  odd}\qquad\mathrm{and}\qquad\mathrm{(b)}\ \ \sigma_k^g,\ \sigma_k^{\bg}>0 \ \ \mathrm{for}\ \ k \ \ \mathrm{is\  even}
    \end{equation*}
    It's equivalent to say $|\sigma_k^g|\leq|\sigma_k^{\bg}|$.

    \ \\

    For even $l$, $\sigma_l^{\bg},\ \sigma_l^g>0$. We assume reversely
    \begin{equation*}
        \int_M\sigma_l^gdv_g\geq\int_M\sigma_l^{\bg}dv_{\bg},
    \end{equation*}
    then     
    \begin{align*}
        \mathcal{F}_{\bg}^{\mathcal{S}}(g)&=\left(\int_M\sigma_l^gdv_g\right)^{2k}\left(\int_M\sigma_k^gdv_{\bg}\right)^{n-2l}
        =(-1)^{nk}\left(\int_M\sigma_l^gdv_g\right)^{2k}\left(\int_M|\sigma_k^g|dv_{\bg}\right)^{n-2l}.
    \end{align*}
    Since $\frac{n}{2}-l<0$, we have 
    \begin{equation*}
        \left(\int_M\sigma_l^gdv_g\right)^{2k}\left(\int_M|\sigma_k^g|dv_{\bg}\right)^{n-2l}\geq\left(\int_M\sigma_l^{\bg}dv_g\right)^{2k}\left(\int_M|\sigma_k^{\bg}|dv_{\bg}\right)^{n-2l}.
    \end{equation*}
    Therefore,
    \begin{align*}
    \mathcal{F}_{\bg}^{\mathcal{S}}(g)
        &\begin{cases}
            \geq\mathcal{F}_{\bg}^{\mathcal{S}}(\bg)\ \  \mathrm{for\ even}\ nk \\
            \leq\mathcal{F}_{\bg}^{\mathcal{S}}(\bg)\ \  \mathrm{for\ odd}\ nk. \\
        \end{cases}
    \end{align*}
    Therefore, by \Cref{Proposition: Regidity},
    \begin{equation*}
        g=c^2\bg
    \end{equation*}
    for constant $c>0$. Now
    \begin{align*}
        \int_M\sigma_l^gdv_g &= c^{n-2l}\int_M\sigma_l^{\bg}dv_{\bg} \geq \int_M\sigma_l^{\bg}dv_{\bg}>0,
    \end{align*}
    which implies $c\leq1$. On the other hand,
    \begin{align*}
        0<|\sigma_k^g| &= c^{-2k}|\sigma_k^{\bg}|\leq|\sigma_k^{\bg}|
    \end{align*}
    which implies $c\geq1$. Therefore, in both cases, we have $c=1$. That is, $g=\bg$.

    \ \\

    For odd $l$, $\sigma_l^{\bg},\ \sigma_l^g<0$. We also assume reversely
    \begin{equation*}
        \int_M\sigma_l^gdv_g\leq\int_M\sigma_l^{\bg}dv_{\bg},
    \end{equation*}
    then    
    \begin{equation*}
        \int_M|\sigma_l^g|dv_g\geq\int_M|\sigma_l^{\bg}|dv_{\bg}.
    \end{equation*}
    Now
    \begin{align*}
        \mathcal{F}_{\bg}^{\mathcal{S}}(g)&=\left(\int_M\sigma_l^gdv_g\right)^{2k}\left(\int_M\sigma_k^gdv_{\bg}\right)^{n-2l} =(-1)^{nk}\left(\int_M|\sigma_l^g|dv_g\right)^{2k}\left(\int_M|\sigma_k^g|dv_{\bg}\right)^{n-2l}.
    \end{align*}
    Similarly, we have 
    \begin{equation*}
        \left(\int_M|\sigma_l^g|dv_g\right)^{2k}\left(\int_M|\sigma_k^g|dv_{\bg}\right)^{n-2l}\geq\left(\int_M|\sigma_l^{\bg}|dv_g\right)^{2k}\left(\int_M|\sigma_k^{\bg}|dv_{\bg}\right)^{n-2l}.
    \end{equation*}
    Therefore,
    \begin{align*}
    \mathcal{F}_{\bg}^{\mathcal{S}}(g)
        &\begin{cases}
            \geq\mathcal{F}_{\bg}^{\mathcal{S}}(\bg)\ \  \mathrm{for\ even}\ nk \\
            \leq\mathcal{F}_{\bg}^{\mathcal{S}}(\bg)\ \  \mathrm{for\ odd}\ nk. \\
        \end{cases}
    \end{align*}
    Therefore, by \Cref{Proposition: Regidity},
    \begin{equation*}
        g=c^2\bg
    \end{equation*}
    for constant $c>0$. Now
    \begin{align*}
        \int_M|\sigma_l^g|dv_g &= c^{n-2l}\int_M|\sigma_l^{\bg}|dv_{\bg} \geq \int_M|\sigma_l^{\bg}|dv_{\bg}>0,
    \end{align*}
    which implies $c\leq1$. Similarly, the relation between $\sigma_k^g$ and $\sigma_k^{\bg}$ implies $c\geq1$. Therefore, in both cases, we have $c=1$. That is, $g=\bg$.

    \ \\

    Now, we focus on the negative Einstein manifold with $k=l=1$. According to the discussion in \Cref{Subsection: 4.1}, we have
    \begin{enumerate}
        \item $n$ is odd, then
        \begin{equation*}
        D^2\mathcal{F}_{\bg}\cdot(h,h) \leq 0;
        \end{equation*}
    \item $n$ is even, then
        \begin{equation*}
            D^2\mathcal{F}_{\bg}\cdot(h,h) \geq 0.
        \end{equation*}
    \end{enumerate}
    Similar to \Cref{Proposition: Regidity}, we have $g=c^2\bg$ provided
    \begin{enumerate}
        \item $n$ is odd and
        \begin{equation*}
        \mathcal{F}_{{\bg}}^{\mathcal{S}}(g)\geq\mathcal{F}_{{\bg}}^{\mathcal{S}}(\bg);
        \end{equation*}
    \item $n$ is even and
        \begin{equation*}
            \mathcal{F}_{{\bg}}^{\mathcal{S}}(g)\leq\mathcal{F}_{{\bg}}^{\mathcal{S}}(\bg).
        \end{equation*}
    \end{enumerate}
    If $R_g\geq R_{\bg}$, since $R_g,\ R_{\bg}<0$, we have $|R_g|\leq |R_{\bg}|$. By assuming reversely
    \begin{equation*}
        \int_MR_gdv_g\geq\int_MR_{\bg}dv_{\bg},
    \end{equation*}
    that is $\int_M|R_g|dv_g\leq\int_M|R_{\bg}|dv_{\bg}$, we have
    \begin{align*}
        \mathcal{F}_{\bg}^{\mathcal{S}}(g)=\left(\int_MR_gdv_g\right)^{2}\left(\int_MR_gdv_{\bg}\right)^{n-2}=(-1)^n\left(\int_M|R_g|dv_g\right)^{2}\left(\int_M|R_g|dv_{\bg}\right)^{n-2}
    \end{align*}
    Therefore,
    \begin{equation*}
        \mathcal{F}_{\bg}^{\mathcal{S}}(g)
        \begin{cases}
            \geq\mathcal{F}_{\bg}^{\mathcal{S}}(\bg)\ \  \mathrm{for\ odd}\ n \\
            \leq\mathcal{F}_{\bg}^{\mathcal{S}}(\bg)\ \  \mathrm{for\ even}\ n.
        \end{cases}
    \end{equation*}
    Hence $g=c^2\bg$ for constant $c>0$. Now
    \begin{align*}
        0<\int_M|R_g|dv_g &= c^{n-2}\int_M|R_{\bg}|dv_{\bg} \leq\int_M|R_{\bg}|dv_{\bg},
    \end{align*}
    which implies $c\leq1$. On the other hand,
    \begin{align*}
        0<|R_g|=c^{-2}|R_{\bg}|\leq |R_{\bg}|
    \end{align*}
    which implies $c\geq1$. Therefore, we have $c=1$. That is, $g=\bg$.
    
\end{proof}

\begin{remark}
    The stability condition is essential. Otherwise, we can construct a counterexample for the case of positive curvature.
\end{remark}

According to \cite{Kroncke}, a prototypical example of an unstable Einstein manifold is the product of two Einstein manifolds with identical positive Einstein constant. Suppose \( n \) is even, and let \( (N^{\frac{n}{2}}_i, g_i) \) for \( i = 1, 2 \) be two Einstein manifolds satisfying
\begin{align*}
    \mathrm{Ric}_{g_i} = \lambda g_i \quad (\lambda > 0).
\end{align*}
Then the product manifold
\begin{align*}
    (M = N_1 \times N_2,\ \bg = g_1 + g_2)
\end{align*}
is an unstable Einstein manifold with
\begin{align*}
    \mathrm{Ric}_{\bg} = \lambda \bg.
\end{align*}

Define a one-parameter family of metrics by
\begin{align*}
g_t = \frac{1}{1+t} g_1 + \frac{1}{1 - t + \alpha t^2} g_2.
\end{align*}
After calculation (see Appendix A), we obtain
\begin{align*}
\sigma_k^{g_t} = &\left( \frac{n-2}{2(n-1)} \lambda \right)^k \binom{n}{k} \left\{ 1 + \frac{k}{2}\phi(k;\alpha)t^2 \right\}+ o(t^2),\\
\int_M \sigma_l^{g_t}\, dv_{g_t}=&\left( 1 + \frac{n-2l}{4}\psi(k,l;\alpha) t^2 \right)\int_M \sigma_l^{\bg}\, dv_{\bg}+ o(t^2),
\end{align*}
where $\phi(k;\alpha)=\alpha-\frac{4(n-1)(k-1)}{(n-2)^2}$ and $\psi(k,l;\alpha)=\frac{n}{n-2l}-\frac{8l(l-1)(n-1)}{(n-2l)(n-2)^2}-\alpha$. 

\ \\

According to Appendix A, when
\begin{align*}
    l\in[1,\frac{2(n-1)+\sqrt{2(n-1)\left(2(n-1)+n(n-2)^2\right)}}{4(n-1)})\subset[1,\frac{n}{2})
\end{align*}
and
\begin{align*}
    k\in[1,\frac{n(n-2)^2-8l(l-1)(n-1)}{4(n-1)(n-2l)}+1)\subset[1,n],
\end{align*}
there exists $\alpha$ such that $\phi(k;\alpha)>0$ and $\psi(k,l;\alpha)>0$, that is 
\begin{align*}
    \sigma_k^{g_t}>\sigma^{\bg}_k\ \ \text{and}\ \ \ \int_M \sigma_l^{g_t}\, dv_{g_t}>\int_M \sigma_l^{\bg}\, dv_{\bg}
\end{align*}
for sufficiently small $|t|$, which gives an counterexample for the case (a) in \Cref{theorem: main_A}.

\begin{remark}
    As for the case (b) in \Cref{theorem: main_A}, the assumption $l>\frac{n}{2}+k$, analogous to a supercritical scenario, poses similar challenges in finding the counterexample. From the perspective of the functional $\mathcal{F}_{\bg}$, the counterexample could be found by taking the perturbed metric $g_t=\bg+th$, where $h\in\mathcal{S}^{TT}_{2,\bg}(M)$ satisfying $\Delta_Eh=\mu h\,(\mu>0)$ and $D^2\mathcal{F}_{\bg}\cdot(h,h)>0$. We have
    \begin{align*}
        D^2\mathcal{F}_{\bg}\cdot(h,h) =a(n,k,l)\lambda^{nk-2}\alpha\mu\left(\mu-\frac{n(n-2)^2}{2\left(n(k-1)-2l(k-l)\right)}\lambda\right)||h||_{L^2(M,\bg)}^2>0
    \end{align*}
    provided $\mu<\frac{n(n-2)^2}{2\left(n(k-1)-2l(k-l)\right)}\lambda$. For unstable product Einstein manifold \[ (M^n=N_1^{n_1}\times N_2^{n_2},\bg=g_1+g_2)\] with $Ric_{\bg}=(n-1)\lambda\bg\, (\lambda>0)$, where $(N_i,g_i)\,(i=1,2)$ are two positive Einstein manifold. We know that $h=n_2g_1-n_1g_2\in\mathcal{S}^{TT}_{2,\bg}(M)$ satisfying $\Delta_Eh=2(n-1)\lambda h$. However, when $l>\frac{n}{2}+k$,
    \begin{align*}
        \frac{n(n-2)^2}{2\left(n(k-1)-2l(k-l)\right)}\leq\frac{n(n-2)^2}{2\left[n(1-1)-2(\frac{n}{2}+1)(1-(\frac{n}{2}+1))\right]}=\frac{(n-2)^2}{n+2}<2(n-1),
    \end{align*}
    which tells that such $h$ is not suitable for the counterexample. Looking for other $h$ or investigating more complicated unstable Einstein manifold, such as $Spin(5)$, is beyond our current capabilities. 
\end{remark}

\begin{remark}
    As for \Cref{theorem: main_B}, the most well known negative Einstein manifold is the hyperbolic space, which is strictly stable. Moreover, the existence of unstable compact negative Einstein manifold is still remain an open problem \cite{Kroncke}. Therefore,the solution to the existence of analogous counterexamples in the setting of \Cref{theorem: main_B} is beyond our current capabilities.
\end{remark}

\ \\

\section*{Appendix A: Calculation of counterexamples}

For 
\begin{align*}
g_t = \frac{1}{1+t} g_1 + \frac{1}{1 - t + \alpha t^2} g_2,
\end{align*}
where $g_i\,(i=1,2)$ is Einstein metric of dimension $\frac{n}{2}\,(n$ is even), satisfying $\mathrm{Ric}_{g_i} = \lambda g_i \quad (\lambda > 0)$. Since $Ric_{g_t}=\lambda \bg$, $\forall t$, we have
\begin{equation*}
    R_{g_t}=\frac{n}{2}\left(2+\alpha t^2\right)\lambda.
\end{equation*}
Therefore, the Schouten curvature of $g_t$ is 
\begin{equation*}
    S_{g_t}=\lambda\left(F_1(t)\frac{g_1}{1+t}+F_2(t)\frac{g_2}{1-t+\alpha t^2}\right),
\end{equation*}
where $F_1(t)=\frac{n-2}{2(n-1)}+t-\frac{n}{4(n-1)}\alpha t^2$ and $F_2(t)=\frac{n-2}{2(n-1)}-t+\frac{3n-4}{4(n-1)}\alpha t^2$. According to the \Cref{eq: sigmak},
\begin{equation*}
    \sigma_k^{g_t}=\lambda^k\sum_{\substack{k_1+k_2=k \\ k_1,\ k_2\leq\frac{n}{2}}}\binom{\frac{n}{2}}{k_1}\binom{\frac{n}{2}}{k_2}F_1^{k_1}(t)F_2^{k_2}(t).
\end{equation*}
Since
\begin{align*}
    F_1^{k_1}=\Bigg(\frac{n-2}{2(n-1)}&\Bigg)^{k_1}\Bigg[1+k_1\left(\frac{n-2}{2(n-1)}\right)^{-1}t\\
    &+\left(\frac{k_1(k_1-1)}{2}\left(\frac{n-2}{2(n-1)}\right)^{-2}-k_1\frac{n}{4(n-1)}\alpha\left(\frac{n-2}{2(n-1)}\right)^{-1}\right)t^2\Bigg]+o(t^2)\\
    =\Bigg(\frac{n-2}{2(n-1)}&\Bigg)^{k_1}\Bigg[1+\frac{2(n-1)}{n-2}k_1t+\left(\frac{2k_1(k_1-1)(n-1)^2}{(n-2)^2}-\frac{n}{2(n-2)}k_1\alpha\right)t^2\Bigg]+o(t^2),\\
    F_2^{k_2}=\Bigg(\frac{n-2}{2(n-1)}&\Bigg)^{k_2}\Bigg[1-k_2\left(\frac{n-2}{2(n-1)}\right)^{-1}t\\
    &+\left(\frac{k_2(k_2-1)}{2}\left(\frac{n-2}{2(n-1)}\right)^{-2}+k_2\frac{3n-4}{4(n-1)}\alpha\left(\frac{n-2}{2(n-1)}\right)^{-1}\right)t^2\Bigg]+o(t^2)\\
    =\Bigg(\frac{n-2}{2(n-1)}&\Bigg)^{k_2}\Bigg[1-\frac{2(n-1)}{n-2}k_2t+\left(\frac{2k_2(k_2-1)(n-1)^2}{(n-2)^2}+\frac{3n-4}{2(n-2)}k_2\alpha\right)t^2\Bigg]+o(t^2),\\
\end{align*}
Then when $k_1+k_2=k$,
\begin{align*}
    F_1^{k_1}F_2^{k_2}=\Bigg(\frac{n-2}{2(n-1)}\Bigg)^{k}\Bigg\{1+&\frac{2(n-1)}{n-2}(k_1-k_2)t\\
    +\Bigg[&-k_1k_2\frac{4(n-1)^2}{(n-2)^2}+\frac{2(n-1)^2}{(n-2)^2}\left(k_1(k_1-1)+k_2(k_2-1)\right)\\
    &+\left((3n-4)k_2-nk_1\right)\frac{\alpha}{2(n-2)}\Bigg]t^2\Bigg\}+o(t^2).
\end{align*}
By the symmetry of $k_1$ and $k_2$, 
\begin{align*}
    \sigma_k^{g_t}=\Bigg(\frac{n-2}{2(n-1)}\lambda\Bigg)^{k}\sum_{\substack{k_1+k_2=k \\ k_1,\ k_2\leq\frac{n}{2}}}\binom{\frac{n}{2}}{k_1}\binom{\frac{n}{2}}{k_2}\Bigg\{1+\Bigg[&-k_1k_2\frac{4(n-1)^2}{(n-2)^2}+\frac{4(n-1)^2}{(n-2)^2}k_1(k_1-1)\\
    &+k_1\alpha\Bigg]t^2\Bigg\}+o(t^2).
\end{align*}
Using the hypergeometric distribution: $X\sim H(n,\frac{n}{2},k)$ with expectation $E(X)=\frac{k}{2}$ and variation $V(X)=\frac{k(n-k)}{4(n-1)}$, 
\begin{itemize}
    \item $\sum_{\substack{k_1+k_2=k \\ k_1,\ k_2\leq\frac{n}{2}}}\binom{\frac{n}{2}}{k_1}\binom{\frac{n}{2}}{k_2}=\binom{n}{k}$;
    \item $\sum_{\substack{k_1+k_2=k \\ k_1,\ k_2\leq\frac{n}{2}}}\binom{\frac{n}{2}}{k_1}\binom{\frac{n}{2}}{k_2}k_1=\binom{n}{k}E(X)$;
    \item $\sum_{\substack{k_1+k_2=k \\ k_1,\ k_2\leq\frac{n}{2}}}\binom{\frac{n}{2}}{k_1}\binom{\frac{n}{2}}{k_2}k_1^2=\binom{n}{k}\left(V(X)+E^2(X)\right)$;
    \item $\sum_{\substack{k_1+k_2=k \\ k_1,\ k_2\leq\frac{n}{2}}}\binom{\frac{n}{2}}{k_1}\binom{\frac{n}{2}}{k_2}k_1k_2=\binom{n}{k}\left(kE(X)-V(X)-E^2(X)\right)$,
\end{itemize}
combining together, we obtain
\begin{align*}
\sigma_k^{g_t} = &\left( \frac{n-2}{2(n-1)} \lambda \right)^k \binom{n}{k} \left\{ 1 + \frac{k}{2}\phi(k;\alpha)t^2 \right\}+ o(t^2)
\end{align*}
where $\phi(k;\alpha)=\alpha-\frac{4(n-1)(k-1)}{(n-2)^2}$. On the other hand, 
\begin{align*}
\int_M \sigma_l^{g_t}\, dv_{g_t^1} =&\left((1+t)\left(1-t+\alpha t^2\right)\right)^{-\frac{n}{4}}\left(1+\frac{l}{2}\phi(l;\alpha)t^2\right)\int_M \sigma_l^{\bg}\, dv_{\bg}+ o(t^2)\\
=&\left(1-\frac{n}{4}(\alpha-1)t^2\right)\left(1+\frac{l}{2}\phi(l;\alpha)t^2\right)\int_M \sigma_l^{\bg}\, dv_{\bg}+ o(t^2)\\
=&\left( 1 + \frac{n-2l}{4}\psi(k,l;\alpha) t^2 \right)\int_M \sigma_l^{\bg}\, dv_{\bg}+ o(t^2),
\end{align*}
where $\psi(k,l;\alpha)=\frac{n}{n-2l}-\frac{8l(l-1)(n-1)}{(n-2l)(n-2)^2}-\alpha$. 

\ \\

Noting
\begin{align*}
    \phi(k;\alpha)=&\alpha-\frac{4(n-1)(k-1)}{(n-2)^2},\\
    \psi(k,l;\alpha)=&\frac{n}{n-2l}-\frac{8l(l-1)(n-1)}{(n-2l)(n-2)^2}-\alpha.
\end{align*}
Denote $\xi(l)=n(n-2)^2-8l(l-1)(n-1)$, then the positive zero of $\xi$ is 
\begin{align*}
    l_0=\frac{2(n-1)+\sqrt{4(n-1)^2+2n(n-1)(n-2)^2}}{4(n-1)}\in(1,\frac{n}{2})
\end{align*}
and $\xi(l)$ decreases in $[1,n]$.

There exists a $\alpha$ such that $\phi(k;\alpha)>0$ and $\psi(k,l;\alpha)>0$ provided
\begin{align*}
    \frac{4(n-1)(k-1)}{(n-2)^2}<\frac{n}{n-2l}-\frac{8l(l-1)(n-1)}{(n-2l)(n-2)^2},
\end{align*}
that is
\begin{align*}
    k<\frac{\xi(l)}{4(n-1)(n-2l)}+1.
\end{align*}
Requiring $k\in[1,n]$ results in
\begin{align*}
    \xi(l)>0,\quad \text{that is}\ \ l<l_0.
\end{align*}
Since 
\begin{align*}
    \xi(l)-4(n-1)^2(n-2l)\leq\xi(\frac{n}{2})-4(n-1)^2(n-n)=\xi(\frac{n}{2})<0,
\end{align*}
we have $\frac{\xi(l)}{4(n-1)(n-2l)}+1<n$. Combining together, when
\begin{align*}
    l\in[1,l_0)\subset[1,\frac{n}{2})\quad\text{and}\quad k\in[1,\frac{\xi(l)}{4(n-1)(n-2l)}+1)\subset[1,n],
\end{align*}
there exists a $\alpha$ such that $\phi(k;\alpha)>0$ and $\psi(k,l;\alpha)>0$.

\end{document}